\documentclass[conference,letterpaper]{IEEEtran}

\makeatletter
 \def\LaTeX{\leavevmode L\raise.42ex
   \hbox{\kern-.3em\size{\sf@size}{0pt}\selectfont A}\kern-.15em\TeX}
\makeatother

\newcommand{\BibTeX}{{\rm B\kern-.05em{\sc
i\kern-.025emb}\kern-.08em\TeX}}
\newtheorem{col}{Corollary}[section]
\newtheorem{thm}{Theorem}[section]
\newtheorem{lem}[thm]{Lemma}
\newtheorem{rem}{Remark}
\newtheorem{defn}{Definition}[section]

\begin{document}

\title{Sampling solutions of Schr\"{o}dinger equations on combinatorial graphs}

\author{\IEEEauthorblockN{Isaac Z. Pesenson}
\IEEEauthorblockA{
 Temple University  \\
Philadelphia, USA\\
Email: pesenson@temple.edu}}

\maketitle

\begin{abstract}

 We consider functions on a graph $G$ whose evolution in time $-\infty<t<\infty$ is governed by a Schr\"{o}dinger type equation with a combinatorial Laplace operator on the right side.  For a given subset $S$ of vertices of $G$ we compute a cut-off frequency $\omega>0$ such that 
 solutions to a Cauchy problem  with initial data in $PW_{\omega}(G)$ are  completely determined by their  samples on $S\times \{k\pi/\omega\},$ where $k\in \mathbf{N}$. It is shown that in the case of a bipartite graph  our results are sharp.

\end{abstract}

\IEEEpeerreviewmaketitle

\section{Introduction}
A sampling theory of bandlimited (Paley-Wiener) functions on combinatorial graphs was initiated in \cite{Pe2008}, \cite{Pe2009}  and currently became a rather active field of research \cite{Pe2010}-\cite{fp}, \cite{NGO}, \cite{SNFOV}. In all of these papers one considered stationary bandlimited signals on graphs. The novelty of the present paper is that we consider non-stationary signals whose evolution is governed by a Schr\"{o}dinger type equation with a combinatorial Laplace operator on the right side. The goal of the paper is to show that solutions of such equations with  bandlimited initial data can be perfectly reconstructed from their samples on the graph and on time axis.

\section{Combinatorial Laplacian  on graphs}

We consider a   graph $G=(V,E)$, where $V = V(G) $ is a countable set of  vertices and $E = E(G)$ is the set of edges or links connecting these vertices.  The weight of the edge connecting two nodes $u$ and $v$ is denoted by $w(u,v)$. The degree $\mu(v)$ of the vertex $v$ is the sum of the edge weights incident to node $v$. The adjacency matrix $W$ of the graph is a matrix such that $W(u,v) = w(u,v)$. 
The Hilbert space $L_{2}(G)$ is the set  of all complex valued functions $f$ on $V(G)$ 
 with the following inner product
\begin{equation}\label{gg0} 
{\langle f,g\rangle}_{L_2(G)}= \langle f,g\rangle = \sum_{v\in V(G)}f(v)\overline{g(v)}\,\mu(v)<\infty.
\end{equation}
The weighted Laplace operator $\Delta$  is introduced  via
\begin{equation}\label{L}
 (\Delta f)(v) = \sum_{u \in V(G)} (f(v)-f(u)) w(v,u)~.
\end{equation}

The  graph Laplacian is a well-studied object; it is known to be a positive-semidefinite self-adjoint  operator. If the set of degrees $\{\mu(v)\}_{v\in V(G)}$ is bounded then the operator $\Delta$ is bounded. 

If the set $V(G)$ is finite then $\Delta$  has $|V(G)|$ real and nonnegative eigenvalues. Since $\Delta$\textbf {1} = 0, where \textbf {1 = (1, 1, \dots, 1)}, is the all 1 constant function, zero is an eigenvalue of $\Delta$ corresponding to the eigenfunction \textbf 1.  Moreover, $0$ is a simple eigenvalue if and only if graph is connected. 
If  $0=\lambda_{0}<\lambda_{1}\leq \dots \leq \lambda_{|V|-1},\>\>\>|V|=|V(G)|,$ is the set of eigenvalues of $\Delta$ the notation  $e_{\lambda_{0}},....,e_{\lambda_{|V|-1}}$ will be used for a corresponding   orthonormal basis of eigenfunctions. For a function $f\in L_{2}(G)$ it's Fourier coefficients $
c_{j}(f)
$ are defined as usual
$$
c_{j}(f)=\sum_{v\in V(G)} f(v)\overline{e_{\lambda_{j}}(v)}.
$$

\section{ Paley-Wiener functions on graphs}

One of the forms of the Spectral Theorem for self-adjoint non-negative operators implies that every function in $L_{2}(G)$ can be identified with a measure on positive semiaxis. 

\begin{defn} (\cite{Pe2008})
We say  that 
a  function $f\in L_{2}(G)$ belongs to a space $PW_{\omega}(G),\>\omega>0,$ (or is  $\omega$-bandlimited)  if the corresponding measure is supported in $[0, \>\omega]$. 
\end{defn}

One can show \cite{Pe2008} that $f\in PW_{\omega}(G)$ if and only if the following Bernstein inequality holds
$$
\|\Delta^{k} f\|\leq \omega^{k}\|f\|,  \>\>k\in \mathbf{N}_{+}.
$$

Note that in the case of a finite graph it means that  fa unction belongs to $PW_{\omega}(G)$  if and only if it is a polynomial in eigenfunctions whose eigenvalues are not greater $\omega$, i. e. 
$$
f=\sum_{\lambda_{j}\leq \omega} c_{j}e_{\lambda_{j}}.
$$

	\section{Sampling solutions of Schr\"{o}dinger equation on time axis}
	We consider the following Cauchy problem 
\begin{equation}\label{C1}
\frac{dg(t, v)}{dt}=i\Delta g(t, v), g(0, v)=f(v), 
\end{equation}
where $v\in V(G),  t\in \mathbf{R}$.

If the initial function $f$ belongs to the domain of the operator $\Delta$ then the unique solution to this problem is given by the formula $g(v,\>t)=e^{it\Delta}f(v),\>\>-\infty<t<\infty,\>\>v\in V(G), $ where $e^{it\Delta}$ is a group of unitary operators in $L_{2}(G)$.

The next theorem is a generalization of what is known as the Valiron-Tschakaloff
sampling/interpolation formula \cite{Pes15}. It shows that if the initial data belongs to $PW_{\omega}(G)$ then the solution to (\ref{C1}) is completely determined by $f,\> \Delta f$ and its 
values at $k\pi/\omega$.
\begin{thm}\label{T100}
For $f\in PW_{\omega}(G),\>\>\>\omega>0,$ we have for all $t  \in \mathbf{R}$
\begin{equation}\label{VT}
g(t,\>v)=
i t\ sinc\left(\frac{\omega  t}{\pi}\right)\Delta f(v)+ sinc\left(\frac{\omega  t}{\pi}\right)f(v)+
$$
$$
\sum_{k\in \mathbf{N},\>k\neq 0}\frac{\omega t}{k\pi} sinc\left(\frac{\omega  t}{\pi}-k\right)g\left(\frac{k\pi}{\omega},\>v\right),
\end{equation}
where convergence is in the space of abstract functions $L_{2}\left((-\infty, \infty), \> L_{2}(G)\right)$ with the regular Lebesgue measure.
\end{thm}
\begin{proof}
If $f\in PW_{\omega}(G)$ then one can show \cite{Pes15} that $e^{it\Delta}f$ is an abstract function with values in $L_{2}(G)$  which is bounded for $t\in \mathbf{R}$ and has extension to complex plane as entire function of exponential type $\omega$. For this reason for any $g\in L_{2}(G)$  the scalar-valued function $F(t)=\left<e^{it\Delta}f, g\right>$ is an entire function of exponential type $\omega$ which is bounded on the real line.  For such functions the so-called  Valiron-Tschakaloff sampling/interpolation formula holds \cite{BFHSS}

\begin{equation}\label{vt}
F(t)=t \> sinc\left(\frac{\omega  t}{\pi}\right)F^{'}(0)+ sinc\left(\frac{\omega  t}{\pi}\right)F(0)+
$$
$$
\sum_{k\neq 0}\frac{\omega t}{k\pi} sinc\left(\frac{\omega  t}{\pi}-k\right)F\left(\frac{k\pi}{\omega}\right).
\end{equation}

The formula (\ref{VT}) now follows from (\ref{vt}) and  the Hahn-Banach theorem. 

\end{proof}

The next Corollary follows from the fact that if $\Delta$ is a bounded operator then every function in $L_{2}(G)$ belongs to $PW_{\omega}(G)$ for any $\omega\geq\|\Delta\|.$

\begin{col}
If  operator $\Delta$ is bounded in the space $L_{2}(G)$ then the previous Theorem holds for every $f\in L_{2}(G)$ as long as $\omega\geq \|\Delta\|$.

\end{col}

\section{Sampling of Paley-Wiener functions on graphs}

\begin{defn} (\cite{Pe2008})
A subset of vertices $S\subset V(G)$ is a sampling set  for a space $PW_{\omega}(G), \> \omega>0$, if there exists two positive constants $c, \>C$ such that
\begin{equation}
c\|f\|^{2}\leq \sum_{s\in S} |f(s)|^{2}\leq C\|f\|^{2}
\end{equation}
\end{defn}

We introduce a measure of "complexity" of a subset of vertices $U\subset V(G)$.
For a subset of nodes $U\subset V(G)$ let  $L_2(U)$ be the subspace of functions in $L_{2}(G)$ which are identical to zero on the compliment $U^{c}=V(G)\setminus U$.

\begin{defn} \label{D1}(\cite{Pe2008})
For a given $\Lambda>0$ we say that $U\subset V(G)$ is $\Lambda$-removable  if any $\varphi \in L_2(U)$ admits a Poincare-type inequality with the constant $\Lambda^{-1} $, i.e., 
\begin{equation}
\|\varphi\|\leq \Lambda^{-1}\|\Delta \varphi\|,\>\>\>\varphi\in L_{2}(U).
\end{equation}
The best constant $\Lambda $ in this inequality is denoted by $\Lambda_{S}$ and called the Poincare constant.
\end{defn}

Note, that since $\lambda_{0}=0$ is an eigenvalue of  $\Delta$ the set of all vertices $V(G)$ is not removable for any $\Lambda>0$.
However, one can show that every proper finite subset of vertices of a graph (finite or infinite) is removable for a certain constant. 

We also note that the Poincare constant $\Lambda_{S}$ is a measure of a size (capacity) of a  set $S$.

The previous  definition is justified by the following Theorem.
\begin{thm}\label{T1}(\cite{Pe2008})
If for a set $S\subset V(G)$   its compliment $S^{c}$ is a $\Lambda_{S^{c}}$-removable then $S$ is a 
sampling set for any space $PW_{\omega}(G)$ with $0< \omega<
\Lambda_{S^{c}}$. In particular, every $f\in PW_{\omega}(G)$ is completely determined by its values on $S$. 
\end{thm}

Note, that this statement can be reformulated as the following  uncertainty principle.

\begin{thm}\label{U1}
If $S\subset V(G)$ is a zero set of a non-trivial  function in $PW_{\omega}(G)$ then $
\omega \Lambda_{S^{c}}>1.
$

\end{thm}

Another reformulation which also has a flavor of an uncertainty principle is the following.

\begin{thm}\label{U2}
For every $S\subset V(G)$ if the intersection $L_{2}(S)\cap PW_{\omega}(G)$ is not trivial then
$$
\omega \Lambda_{S}>1.
$$
\end{thm}

We are going to obtain a lower bound of $\Lambda_{S}$ for any set $S$ by using only the geometry of  a graph. Given any subset $S \subset V(G)$, we set
$
w_S(v) = \sum_{s \in S} w(s,v),\>\>\>v \in V(G),
$
 and then introduce the following "measures of connectivity" between $S$ and its compliment $S^{c}$: 
$$
 D_{S}  =  D_{S\rightarrow S^{c}} = \sup_{s \in S} w_{S^{c}}(s),
 $$
 $$
K_S = K_{S\leftarrow S^{c}} = \inf_{v \in S^{c}} w_{S}(v).
$$
The next Theorem follows from a more general estimate which is proved  in \cite{PePe2010}, \cite{fp}.

\begin{thm} (\cite{PePe2010}, \cite{fp})\label{thm:main}
For any $S\subset V(G)$ and any $f\in L_{2}(G)$ the following inequality holds 
\[
 \| f \| \leq  \left(K_{S}\right)^{-1/2}\| \Delta^{1/2} f \| \>+  \left(\frac{D_{S}}{K_{S}}\right)^{1/2}\| f|_{S^{c}} \|.
\]
\end{thm}
From here we obtain that for any $\varphi\in L_{2}(S)$
$$
 \| \varphi\| \leq  \left(K_{S}\right)^{-1}\| \Delta \varphi \|,\>\>\>\varphi\in L_{2}(S).
$$

The optimality of the Poincare constant implies the inequality
$
\Lambda_{S}\geq K_{S}.
$
\begin{col}
Every $S\subset V(G)$ is a sampling  set for all functions in $PW_{\omega}(G)$ with any $\omega<K_{S}$. 
\end{col}

The next Corollary follows from the general theory of Hilbert frames \cite{Gr}.
\begin{col} \label{cor3}
If $S$ is a sampling set for the space $PW_{\omega}(G)$ then there exists a frame $\{\Phi_{s}\}_{s\in S}$ in the space $PW_{\omega}(G)$ such that for any $f\in PW_{\omega}(G)$ the following reconstruction formula holds
\begin{equation}
f(v)=\sum_{s\in S}f(s)\Phi_{s}(v),\>\>\>v\in V(G).
\end{equation}

\end{col}

Note, that the previous three  statements hold  for both finite and infinite graphs.

\section{Finite graphs}
We now consider a finite graph of $|V|=|V(G)|$ vertices. 
In what follows the notation $\mathcal{N}[a,\>b]$ is used for the number of eigenvalues of $\Delta$ in $[a,b]$.  The following Corollary follows from the previous theorem.
\begin{col}\label{col} (\cite{Pe2008}, \cite{fp})
For every  set $S\subset V(G)$  the following holds
\begin{enumerate}

\item the Poincare constant of its compliment satisfies the inequality $\Lambda_{S}\geq  K_{S}=\inf_{v \in S^{c}} w_{S}(v)$;

\item it is a uniqueness set for every space $PW_{\omega}(G)$ with $\omega<\Lambda_{S}$;

\item  $\mathcal{N}\left [0, \>\>    \Lambda_{S}   \right)\leq|S|$;

\item  $ \mathcal{N}\left [\Lambda_{S} ,\>\>\lambda_{|V|-1}\right]\geq |S^{c}|$;

\item  $
\lambda_{|S|}\geq   \Lambda_{S}$ .

\end{enumerate}
\end{col}

The following test for cut-off frequency can be used for finite graphs.

\begin{thm}\label{coff}
Let $S\subset V(G) $ and $S^{c} =V(G)\setminus S$. Let $\mathcal{L}$  be a matrix which
is obtained from the matrix of $\Delta$   by replacing by zero columns and rows corresponding to the set $S$.
Then $S$ is a sampling set for all signals $f \in  PW_{\omega}(G)$  with $\omega<\sigma$ where $\sigma$ is the smallest positive eigenvalue of $\mathcal{L}$.

\end{thm}

\begin{proof}
According to Definition \ref{D1} and Theorem \ref{T1} it suffices to show that $S^{c}$ is a $\Lambda$-set for $\Lambda=1/\sigma$, i.e., for any $\varphi\in L_{2}(G)$ whose restriction to $S$ is zero one has 
\begin{equation}
\|\varphi\|\leq \frac{1}{\sigma}\|\Delta\varphi\|,\>\>\>\varphi|_{S}=0.
\end{equation}

Note, that since matrix of $\Delta$ is symmetric and since in the later we replace by zero columns and rows with the same set of indices the matrix $\mathcal{L}$ is also symmetric. Since $\Delta$ is non-negative the matrix $\mathcal{L}$ is also non-negative. It follows from the fact that if $\varphi$ belongs to the subspace $L_{2}(S^{c})$ of all $\varphi\in L_{2}(G)$ such that $\varphi|_{S}=0$ then 
$$
\left<\mathcal{L} \varphi, \varphi\right>=\left<\Delta \varphi, \varphi\right>\geq 0.
$$
\begin{rem}
Warning: the equality
$$
\left<\mathcal{L} \varphi, \varphi\right>=\left<\Delta \varphi, \varphi\right>
$$
holds for every $\varphi$ in the space $L_{2}(S^{c})$, but in general even for functions $\varphi$ in $L_{2}(S^{c})$ there is no equality 
$$
\mathcal{L}\varphi=\Delta\varphi.
$$
\end{rem}

Now we are going to show that $\mathcal{L}$ is strictly positive on $L_{2}(S^{c})$. Indeed, it is
clear that the subspace $L_{2}(S^{c})$ is invariant with respect to $\Delta$. This fact allows to identify $\mathcal{L}$ with an operator  in  $L_{2}(S^{c})$.
If $\varphi\in L_{2}(S^{c})$ is not identical zero but $\Delta\varphi=0$ then
$$
0=\left<\mathcal{L} \varphi, \varphi\right>=\left<\Delta \varphi, \varphi\right>=\left<\Delta^{1/2}\varphi, \Delta^{1/2}\varphi\right>=\|\Delta^{1/2}\varphi\|^{2},
$$
which implies that $\Delta^{1/2}\varphi=0$ and then $\Delta\varphi=0$. as the formula 
$$
\Delta\varphi(v)=\sum_{u\sim v}w(u,v)(\varphi(v)-\varphi(u))
$$
shows, only functions which are constant on entire graph belong to the kernel of $\Delta$.  Since constants do not belong to $L_{2}(S^{c})$ it implies strict positivity of the operator $\mathcal{L}$ on $L_{2}(S^{c})$.

Let $0<\sigma=\sigma_{0}\leq \sigma_{1}\leq...\leq \sigma_{m-1},\>\>m=|S^{c}|$, be the set of eigenvalues of $\mathcal{L}$ counting with their multiplicities and $e_{0},..., e_{m-1}$ be the corresponding
set of orthonormal eigenvectors that forms a basis in $L_{2}(S^{c})$.

For $\varphi\in \Delta$ we have
\begin{equation}\label{*}
\frac{   \|\Delta^{1/2}\varphi\|^{2} }{\|\varphi\|^{2}}=\frac{  \left<\Delta^{1/2}\varphi, \Delta^{1/2}\varphi\right>  }{\|\varphi\|^{2}}=\frac{   \left<\Delta \varphi, \varphi\right> }{\|\varphi\|^{2}}=\frac{ \left<\mathcal{L} \varphi, \varphi\right>   }{\|\varphi\|^{2}}.
\end{equation}
If $\varphi-\sum_{j=0}^{m-1}c_{j}e_{j}$ where $c_{j}=\left<\varphi, e_{j}\right>$, then

$$
\mathcal{L}\varphi=\sum_{j=0}^{m-1}\sigma_{j}c_{j}e_{j}
$$
and by Parseval equality
$$
\left<\mathcal{L}\varphi,\varphi\right>=\sum_{j=o}^{m-1}\sigma_{j}|c_{j}|^{2}\geq \sigma\|\varphi\|^{2},
$$
where $\sigma$ is the smallest eigenvalue of $\mathcal{L}$.

This inequality along with (\ref{*}) imply that for any $\varphi$ whose restriction to $S$ is zero we have
$$
\frac{   \|\Delta^{1/2}\varphi\|^{2} }{\|\varphi\|^{2}}\geq \sigma,\>\>\varphi\in L_{2}(S^{c}),
$$
or
\begin{equation}
\|\varphi\|\leq\frac{1}{\sqrt{\sigma}}\|\Delta^{1/2}\varphi\|, \>\>\varphi\in L_{2}(S^{c}).
\end{equation}
From here we obtain  the inequality
\begin{equation}
\|\varphi\|\leq\frac{1}{\sigma}\|\Delta\varphi\|, \>\>\varphi\in L_{2}(S^{c}).
\end{equation}
In other words, $S^{c}$ is a $\Lambda$-set with $\Lambda=1/\sigma$. Therefore, by Theorem \ref{T1}, $S=V(G)\setminus S^{c}$ is a uniqueness set for all signals $\varphi \in PW_{\omega}(G)$ with $\omega<\sigma$. Theorem is proved.
\end{proof}
It should be noted that the above Theorem was motivated by a similar result obtained in \cite{NGO}.

\section{Example of a bipartite graph} 
A non-weighted complete bipartite graph $G$ consists of two disjoint sets of vertices  (components)  $S$ and $S^{c}$ where 
	
	\begin{enumerate}
	
	\item every vertex in one component is connected to every vertex in another;
	
	\item  no edges inside of components;
	
	\item  every edge has weight one.
	
	\end{enumerate} 
Let $|S|=N$,  $\>\>|S^{c}|=M$, and $\>\>N>M$.
 One can show that in this case  the Laplacian $\Delta$ on graph $G$  has eigenvalues $\lambda_{0},...,\lambda_{M+N-1}$, where
$$
\lambda_{0}=0<\lambda_{1}=...=\lambda_{N-1}=M<
$$
$$
\lambda_{N}=...=\lambda_{N+M-2}=N<\lambda_{N+M-1}=N+M.
$$

One can verify that the following statement holds.

\begin{lem}\label{Ex}
In the case of a bipartite graph  all the statements of the  Corollary \ref{col} are sharp and we have:
\begin{enumerate}

\item for the Poincare constant of $S^{c}$  we have $
\Lambda_{S}=K_{S}=|S|=N;
$

\item the set $S$ is the uniqueness set for the span of the first $N$ eigenfunctions $e_{0},...,e_{N-1}$, where $\>\>\Delta e_{j}=\lambda_{j}e_{j}$;

\item  $\mathcal{N}\left [0, \>\>    N   \right)= N$;

\item  $ \mathcal{N}\left [N ,\>\>\lambda_{N+M-1}\right]= M$;

\item  $
\lambda_{N}=   N$ .

\end{enumerate}
	
\end{lem}	
\section{Main result}

The next Theorem is a consequence of the Theorem \ref{T100} and Corollary \ref{cor3}.

\begin{thm}\label{M}
Assume that $f\in PW_{\omega}(G)$ and  $S\subset V(G)$ is a sampling  set for $PW_{\omega}(G)$.  Then the solution $g(t,  v)$ to the Cauchy problem (\ref{C1}) at any point $(t, v)\in \mathbf{R}\times V(G)$ is completely determined by the values of $\Delta f$ on $S$ and by the set of samples $g(k\pi/\omega, s), k\in \mathbf{N}\cup{\{0\}}, s\in S$. Moreover, the explicit reconstruction formula is given by
\begin{equation}
g(t, v)=
$$
$$
\sum_{s\in S}\left(it\ sinc\left(\frac{\omega  t}{\pi}\right)\left(\Delta f\right)(s)+
sinc\left(\frac{\omega  t}{\pi}\right)g(0, s)\right) +
$$
$$
\sum_{s\in S}\left(\sum_{k\in \mathbf{N}, k\neq 0}\frac{\omega t}{k\pi}  sinc\left(\frac{\omega  t}{\pi}-k\right)g\left(\frac{k\pi}{\omega}, s\right)\right)\Phi_{s}(v), 
 \end{equation}
where $ g(0, s)=f(s)$ and $\{\Phi_{s}\}_{s\in S}$ is a frame in $PW_{\omega}(G)$ described in Corollary \ref{cor3}.

\end{thm}

\section{Conclusion}

In  the present paper we consider non-stationary signals which propagate on a combinatorial graph and whose evolution is governed by a Schr\"{o}dinger type equation with a combinatorial Laplace operator on the right side.  It is shown that such signals  can be perfectly reconstructed from their samples on the graph and on the time axis. The main new result of the paper is the Theorem  \ref{M} which relies on the other new fact formulated in Theorem \ref{T100}.

Theorems \ref{thm:main} and \ref{coff} contain  tests for estimating a cut-off frequency of a given sampling set. It is shown in Lemma \ref{Ex} that Theorem \ref{thm:main} cannot be improved in general.

In Theorems \ref{U1} and \ref{U2} two new forms of uncertanty principle on graphs are formulated.

\end{document}